\documentclass[12pt,a4paper]{amsart}
\usepackage[T1]{fontenc}
\usepackage[utf8]{inputenc}
\usepackage{lmodern}

\usepackage{mathtools}

\usepackage{enumitem}
\usepackage{amsthm}
\usepackage{amssymb}
\usepackage{xcolor}
\usepackage{hyperref}

\newtheorem{thm}{Theorem}
\newtheorem{lem}[thm]{Lemma}
\newtheorem{cor}[thm]{Corollary}
\newtheorem{pro}[thm]{Proposition}

\newtheorem{prob}{Problem}

\newcommand{\bbR}{\mathbb R}
\newcommand{\N}{\mathbb N}
\newcommand{\diam}{\mathrm{diam}\,}
\newcommand{\inte}{\mathrm{int}\,}

\renewcommand{\leq}{\leqslant}
\renewcommand{\geq}{\geqslant}

\date{}
\title{Kenyon theorem revisited}

\author[S. G\l \c{a}b]{Szymon G\l \c{a}b}
\address{Institute of Mathematics, \L \'od\'z University of Technology, al. Politechniki 8, 93-590 \L \'od\'z, Poland}
\email{szymon.glab@p.lodz.pl}

\author[M. Kula]{Mateusz Kula}
\address{Institute of Mathematics, University of Silesia in Katowice, Bankowa 14, 40-007 Katowice}
\email{mateusz.kula@us.edu.pl}

\begin{document}

\begin{abstract}
    We study sets $E(\Sigma,q)=\left\{\sum_{i=1}^\infty \sigma_iq^i\colon(\sigma_i)\in\Sigma^\N\right\}$ for a finite set $\Sigma\subset \bbR$ and $q\in(0,1)$. Under the assumption $q|\Sigma|=1$ we prove several new equivalent conditions for $E(\Sigma,q)$ to contain an interval. We give a full characterization, if additionally $|\Sigma|$ is prime.
\end{abstract}
\maketitle
\section{Introduction}
\label{sec1}
Let
$$E(\Sigma,q):=\left\{\sum_{i=1}^\infty \sigma_iq^i\colon(\sigma_i)\in\Sigma^\N\right\}$$ for some finite set $\Sigma\subset \bbR$ and $q\in(0,1)$. Sets of this form arise naturally in the study of achievement sets as well as infinite convolutions of discrete measures, in particular Bernoulli convolutions, see for example \cite{Varju, Solomyak, Solomyak2}. Achievement sets were introduced by Kakeya \cite{Kakeya} and their properties have been deeply studied since then, see for example \cite{BBGS, Jones, Nitecki}.
In a Banach space, the \textit{achievement set} (or \textit{set of subsums}) of an absolutely convergent series $\sum_n x_n$ is the set
$$A(x_n):=\left\{\sum_{n\in B}x_n\colon B\subset \mathbb N\right\}.$$ By the theorem of Guthrie and Nymann \cite{GN} (see Theorem \ref{thm1}) achievement sets in $\mathbb R$ are characterized topologically: they are finite sets,  finite unions of intervals, homeomorphic copies of the Cantor set or Cantorvals. There is an equivalent condition for an achievement set to be a finite union of intervals, see Proposition \ref{pro1}. However, the problem to decide whether an achievement set is a Cantor set or a Cantorval remains open for many concrete series. Achievement sets of multigeometric sequences, i.e. sequences of the form $$(p_1,\ldots,p_m;q):=(p_1q,\ldots,p_mq,p_1q^2,\ldots, p_mq^2,\ldots)$$
for some real numbers $p_1,\ldots, p_n$ and $q\in(0,1)$, are special cases of sets $E(\Sigma,q)$. That is why a problem of determining the topological type of an achievement set of a  multigeometric sequence is reduced to checking whether the corresponding set $E(\Sigma,q)$ contains an interval.

Throughout the paper $\lambda$ denotes one-dimensional Lebesgue measure, $\dim_H$ is the Hausdorff dimension, $\diam A$ is the diameter of $A$ and $|A|$ is the  cardinality of $A$. For us
a \textit{Cantor set} is any subset $C\subset\bbR$ that is homeomorphic to the ternary Cantor set.

Due to \cite{BBFS} we know, for example, that
\begin{itemize}
    \item $E(\Sigma,q)$ is a Cantor set if $q<\frac{1}{\vert\Sigma\vert}$;
    \item $E(\Sigma,q)$ is an interval if and only if $q\geq I(\Sigma):=\frac{\Delta(\Sigma)}{\Delta(\Sigma)+\text{diam}\Sigma}$, where $\Delta(A)=\max\{b-a \colon a,b\in A, a<b, (a,b)\cap A=\emptyset\}$;
    \item \sloppy $E(\Sigma,q)$ contains an interval if $q\geq i(\Sigma)$, where $i(A)=\min\{I(B)\colon B\subset A, 2 \leq |B|< \omega\}$.
\end{itemize}
\fussy
We refer to \cite{BBFS} for more results of this form. In the case when $q=\frac{1}{\vert\Sigma\vert}$ Nitecki \cite{Nitecki} proved that if $0\in\Sigma\subseteq\N\cup\{0\}$ and all remainders from dividing elements of $\Sigma$ by $\vert\Sigma\vert$ are pairwise distinct, then $E(\Sigma,q)$ contains an interval. Nitecki mentioned that what he proved was based on Kenyon's idea from \cite{K}. Kenyon was interested in  projections of one-dimensional Sierpi\'nski gasket, which led him to studying sets of the form $E(\Sigma,\frac13)$ for $\Sigma=\{0,1,u\}$. Here we generalize his ideas and we get the characterization of those finite sets $\Sigma\subset\bbR$ for which $E(\Sigma,\frac{1}{\vert\Sigma\vert})$ contains an interval.

The paper is organised as follows. We finish this section with relevant results, which served us as motivation. 
In section \ref{sec2} we intodroduce the notions and lemmas we need for the main theorem, which we prove in section \ref{sec3}. In section \ref{sec4} we deal with corollaries, in particular the case when $|\Sigma|$ is prime. 

\subsection{Achievement sets}
The study of achievement sets of arbitrary absolutely convergent series can be reduced to non-increasing sequences of positive terms, see for example \cite{BFP}.
\begin{pro}[Kakeya \cite{Kakeya}]\label{pro1}
    Let $(x_n)$ be a non-increasing sequence with positive terms  such that $\sum_nx_n$ is convergent. Then
    \begin{enumerate}[label=(\roman*)]
        \item $A(x_n)$ is a compact perfect set;
        \item $x_n>\sum_{k=n+1}^\infty x_k$ for almost all $n$ implies that $A(x_n)$ is a Cantor set;
        \item $x_n\leq \sum_{k=n+1}^\infty x_k$ for almost all $n$ if and only if $A(x_n)$ is a finite union of closed intervals.
    \end{enumerate}
\end{pro}
Kakeya conjectured that finite unions of intervals and Cantor sets are the only possible forms of achievement sets in $\bbR$.
Later on, it was observed by several authors \cite{F, GN, WS} that there is another possibility.  
 Let us recall the construction of ternary Cantor set. In the first step from the unit interval $[0,1]$ we remove $U_1=(\frac13,\frac23)$, and obtain a union $[0,\frac13]\cup[\frac23,1]$. In the second step from $[0,\frac13]\cup[\frac23,1]$ we remove $U_2:=(\frac19,\frac29)\cup(\frac79,\frac89)$, and obtain $[0,\frac19]\cup[\frac29,\frac13]\cup[\frac23,\frac79]\cup[\frac89,1]$, etc. Then $[0,1]\setminus\bigcup_{n=1}^\infty U_n$ is the ternary Cantor set. 

If one removes from $[0,1]$ only $U_{n}$ for $n$ even, then one may consider the following set $C:=[0,1]\setminus\bigcup_{n=1}^\infty U_{2n}$. This Cantor-like construction leads to the notion of Cantorval (sometimes called the middle Cantorval or $M$-Cantorval).  A \textit{Cantorval} is a homeomorphic copy of the set $C$ defined above. 
The following topological characterization of achievement sets shows that there are only three possible forms of them. 

\begin{thm}[Guthrie, Nymann, S\'aenz \cite{GN,NS}] \label{thm1}
For any non-increasing sequence with positive terms $(x_n)$ such that $\sum_nx_n$ is convergent, the achievement set $A(x_n)$ is one of the following sets:
\begin{enumerate}
    \item a finite union of closed intervals;
    \item a Cantor set;
    \item a Cantorval.
\end{enumerate}
\end{thm}
The following relationship exists between achievement sets of multigeometric sequences and sets $E(\Sigma, q)$. Given a multigeometric sequence $(x_n)=(p_1,\ldots,p_m;q)$ for $$\Sigma=\left\{\sum_{i=1}^m\varepsilon_ip_i\colon (\varepsilon_i)\in \{0,1\}^m\right\}$$ we have $A(x_n)=E(\Sigma,q)$.
The following result, essentially due to Nitecki \cite{Nitecki}, was our main motivation for this paper.

\begin{thm}
    Let $(k_1,\dots,k_m;\frac{1}{n})$ be a mutigeometric series such that $n=\vert\Sigma\vert$. Let $\{\sigma_0<\sigma_1<\dots<\sigma_{n-1}\}$ be an enumeration of $\Sigma$. Let $t_0,t_1,\dots,t_{n-1}\in\{0,1,\dots,n-1\}$ be such that $\sigma_i\equiv t_i\mod n$.  Assume that $\{t_0,t_1,\dots,t_{n-1}\}=\{0,1,\dots,n-1\}$. Then 
\begin{itemize}
    \item[(i)] If $\sigma_i=i\sigma_1$, then $A(k_1,\dots,k_m;\frac{1}{n})$ in an interval $\left[0,\frac{n}{n-1}\sum_{i=1}^mk_i\right]$;
    \item[(ii)] otherwise $A(k_1,\dots,k_m;\frac{1}{n})$ in a Cantorval. 
\end{itemize}
\end{thm}

We consider the problem, if this result can be reversed. More precisely, suppose that $(k_1,\dots,k_m;\frac{1}{n})$ is a mutigeometric series such that $n=\vert\Sigma\vert$ and $A(k_1,\dots,k_m;\frac{1}{n})$ has positive Lebesgue measure. Is is true that there is a constant $c>0$ such that $\Sigma\subset c\mathbb Z$ and the set of remainders $t_i<n$ with $\frac{1}{c}\sigma_i\equiv t_i$ equals $\{0,1,\dots,n-1\}$? We will prove that this is true for every prime number $n>2$. We also present an example that it is not true for $n=4$. Finally, we propose a conjecture what condition needs to be added to get full characterization.

\section{Preliminaries} \label{sec2}

We say that a family $\mathcal T$ is an $E$-\textit{tiling} of  $A\subset \mathbb R$, whenever $A\subset \bigcup \mathcal T$ and $\mathcal T$ consists of measure disjoint translations of $E$. If $E$ is bounded, then any $E$-tiling uniquely determines the set $B\subset \mathbb R$ such that
$$\mathcal T=\{\beta+ E:\beta\in B\}.	$$
We shall refer to this set $B$ as the \textit{basis} for the tiling $\mathcal T$, see also \cite{LW}.
We say that $B$ is \textit{periodic}, if there exists $t\neq 0$ such that $B=t+B$. An $E$-tiling is \textit{periodic}, if its basis is periodic.

If $S\subset \mathbb R$ is a set, then denote by $D(S)$ the set of all differences of numbers in $S$, that is,
$$D(S)=\{s-s'\colon s,s'\in S\}.$$
We say that $\delta\in\mathbb R$ is the \textit{common divisor} of the set $S\subset \mathbb R$, whenever $S\subset \delta\mathbb Z$. If the set $S\subset \mathbb R$ has a common divisor, then there exists the greatest common divisor of $S$.

Let $q\in(0,1)$ and   $\Sigma\subset \mathbb R$ be a finite set. For each $n\in\mathbb N$ consider the Borel measure $\nu_n\colon \mathcal B (\mathbb R)\rightarrow [0,1]$ given by the formula
$$\nu_n(B)=\frac{|B\cap (q^n\Sigma )| }{|\Sigma|}=\frac1{|\Sigma|}\sum_{\sigma\in\Sigma}\mathbf{1}_{B}(\sigma q^n),$$ and define measure $\mu_n\colon\mathcal B (\mathbb R)\rightarrow [0,1]$ as the convolution $$\mu_n=\nu_1\ast\dots\ast\nu_n,$$
which is the same as
$$\mu_n(B)=\frac1{|\Sigma|^n}{\sum_{(\sigma_i)\in\Sigma^n}\mathbf{1}_{B}\left(\sum_{i=1}^n\sigma_iq^i\right)}.$$
We also have
$$\mu_{n+1}(B)=\frac1{|\Sigma|}\sum_{\sigma\in\Sigma}\mu_n\left(\frac1q B- \sigma\right).$$
The sequence of measures $(\mu_n)$ converges weakly to some measure $\mu_{\infty}$, which is supported on $E(\Sigma,q)$. We will say that $\mu_{\infty}$ is \textit{associated} with $\Sigma$ and $q$.

We finish this section with some lemmas, which we will need later on. The first one appears to be folklore, but we present the proof for completeness.
\begin{lem} Let $A_1,\ldots, A_n\subset \mathbb R$ be Lebesgue measurable sets of finite measure.
	Then the function $f\colon \mathbb R^n\rightarrow \bbR$ given by the formula
	$$f(x_1,\ldots,x_n)=\lambda\left(\bigcap_{i=1}^n(x_i+ A_i)\right)$$
	is continuous. In particular, for any Lebesgue measurable sets $A,B\subset\bbR$ of finite measure, the function $g\colon \mathbb R^k\rightarrow [0,\infty)$ given by the formula
	$$g(x_1,\ldots, x_k)=\lambda\left(\bigcup_{j=1}^k(x_j+A)\cap B\right)$$
	is continuous.
	\end{lem}
	\begin{proof}
		If the sets $A_i$ are open intervals, then
		$$\lambda\left(\bigcap_{i=1}^n(x_i+ A_i)\right)=\max\left\{0, \min_{1\leq i\leq n} \left(x_i+\sup A_i\right)-\max_{1\leq i\leq n}\left( x_i+\inf A_i\right) \right\},$$
		and it is a continuous function.
		
		Now we show by induction on $m = 0,1,\ldots, n$ that
		if $A_1, \ldots, A_m$ are open and $A_{m+1},\ldots, A_n$ are open intervals, then $f$ is continuous. Fix $m$ and assume that the hypothesis holds for $m-1$. Then
		$A_m=\bigcup_{k=1}^\infty I_k^m$ for some pairwise disjoint open intervals $I_k^m$ and
		$$\lambda\left(\bigcap_{i=1}^n(x_i+ A_i)\right)=\sum_{k=1}^\infty \lambda\left( \bigcap_{i=1}^{m-1}(x_i+A_i)\cap (x_m+ I_k^m)\cap \bigcap_{i=m+1}^n\left(x_i+A_i\right)\right).$$
		Hence $f$ is a sum of a series of functions that are continuous by induction hypothesis and the series in uniformly convergent by Weierstrass criterion ($\sum_{k=1}^\infty \lambda(I^m_k)$ is the dominating convergent series). This finishes the induction proof and hence we showed that $f$ is continuous, whenever $A_1,\ldots, A_n$ are open sets.
		
		For the general case, fix measurable sets $A_1,\ldots, A_n\subset \mathbb R$ of finite measure, a point $x^0=(x_1^0,\ldots, x_n^0)\in\bbR^n$ and $\varepsilon>0$. Take open sets $U_i\supset A_i$ such that $\lambda(U_i\setminus A_i)<\frac\varepsilon{3n}$. For all $(x_1,\ldots, x_n)\in\bbR^n$ we have
		$$\bigcap_{i=1}^n(x_i+U_i)\subset \bigcap_{i=1}^n(x_i+A_i)\cup \bigcup_{i=1}^n\left(x_i+U_i\setminus A_i\right)$$
		and hence
		$$\lambda\left(\bigcap_{i=1}^n(x_i+U_i)\right)-\lambda\left(\bigcap_{i=1}^n(x_i+A_i)\right)\leq \sum_{i=1}^n \lambda(U_i\setminus A_i)< \frac\varepsilon3 .$$
		For all $x\in \bbR^n$ sufficiently close to $x^0$ we have
		\begin{multline*}
			|f(x^0)-f(x)|\leq \left|\lambda\left(\bigcap_{i=1}^n(x_i^0+A_i)\right)-\lambda\left(\bigcap_{i=1}^n(x_i^0+U_i)\right)\right| +\\+\left|\lambda\left(\bigcap_{i=1}^n(x_i^0+U_i)\right)-\lambda\left(\bigcap_{i=1}^n(x_i+U_i)\right)\right|+\\+\left|\lambda\left(\bigcap_{i=1}^n(x_i+U_i)\right)-\lambda\left(\bigcap_{i=1}^n(x_i+A_i)\right)\right|<\frac\varepsilon3+\frac\varepsilon3+\frac\varepsilon3=\varepsilon.
		\end{multline*}
The additional part is a consequence of the inclusion-exclusion principle.
		\end{proof}
If $f\colon X\rightarrow Y$ is a measurable function between measure spaces $(X,\mathcal A, \mu)$ and $(Y,\mathcal B, \nu)$   such that $\nu(B)=\mu(f^{-1}[B])$ for all $B\in \mathcal B$, then, as usual,  $\nu$ is called the \textit{pushforward measure} and we denote it by $f\#\mu$.

Let $(X,\mathcal A, \mu)$ be a measure space. The measurable map $T\colon X\rightarrow X$ is \textit{measure-preserving}, if  $\mu=T\#\mu$. A measure-preserving map $T\colon X\rightarrow X$ is \textit{ergodic}, if for any $A\in\mathcal A$,
$$T^{-1}[A]=A \implies \mu(A)\in\{0,1\}.$$

\begin{lem}
    \label{lem3} Let $X$ be compact Hausdorff space and let $\mu$, $\nu$ be Borel probability measures on $X$. If there exists a map $T\colon X\rightarrow X$ that is measure-preserving and ergodic with respect to both $\mu$ and $\nu$, then $\mu=\nu$.
\end{lem}
\begin{proof}
    For any continuous function $f\colon X\rightarrow \bbR$, by Birkhoff Ergodic Theorem, see \cite[Theorem 2.30]{EinsiedlerWard}, $$\int_Xf\,\mathrm{d}\mu=\lim_{n\to\infty}\frac1n\sum_{j=0}^{n-1}f(T^j(x))=\int_Xf\,\mathrm{d}\nu$$ almost everywhere.
\end{proof}
\begin{lem}\label{lem4}
    Let $\mu$ and $\nu$ be Borel probability measures on $[0,1)$. If for all $k\in\mathbb Z$,
    $$\int_0^1\exp(2\pi i k s)\mathrm d \, \mu (s)=\int_0^1\exp(2\pi i k s)\mathrm d \, \nu (s)$$
    (i.e. all respective Fourier coefficients are equal), then $\mu=\nu$.
\end{lem}
\begin{proof}
    It follows from the Stone-Weierstrass theorem that the set of linear combinations of $\{e_n\colon n\in \mathbb Z\}$, where $e_n(x)=\exp(2\pi i n x)$, is dense in the space of complex valued continuous functions on $X=[0,1]/\{0,1\}$ with supremum norm. Here $[0,1]/\{0,1\}$ denotes the quotient space of $[0,1]$ with $0$ and $1$ identified. Hence for any continuous function $f\colon X\rightarrow \mathbb C$, there exists a sequence $f_n\in \operatorname{lin}\{e_n\colon n\in \mathbb Z\}$ such that $(f_n)$ uniformly converges to $f$, and consequently
\begin{multline*}
\int_Xf\mathrm{d}\,\mu=\int_X\lim_{n\to\infty}f_n\mathrm{d}\,\mu=\lim_{n\to\infty}\int_Xf_n\mathrm{d}\,\mu=\\=\lim_{n\to\infty}\int_Xf_n\mathrm{d}\,\nu=\int_X\lim_{n\to\infty}f_n\mathrm{d}\,\nu=\int_Xf\mathrm{d}\,\nu.
\end{multline*}
\end{proof}

\begin{lem}[\cite{Nitecki}]\label{lem2}
    Assume $S\subset \mathbb Z$ is a finite set of cardinality $m$ such that the remainders modulo $m$ of numbers in $S$ are pairwise distinct. Given two sequences $(\sigma_1,\ldots, \sigma_n), (\sigma_1',\ldots, \sigma_n') \in S^n$, if the number
    $$\sum_{i=1}^n\frac{\sigma_i}{m^i}-\sum_{i=1}^n\frac{\sigma_i'}{m^i}$$ is an integer, then $\sigma_i=\sigma_i'$ for all $i\in\{1,\ldots,n\}$.
\end{lem}
\begin{proof}
    The proof is by induction on $n$. Fix $n\in\mathbb N$ and assume the hypothesis holds for $n-1$. If $$\sum_{i=1}^n\frac{\sigma_i}{m^i}-\sum_{i=1}^n\frac{\sigma_i'}{m^i}\in\mathbb Z,$$ then $$\sigma_1-\sigma_1'+\sum_{i=2}^n\frac{\sigma_i}{m^i}-\sum_{i=2}^n\frac{\sigma_i'}{m^i} \in m\mathbb Z.$$ By induction hypothesis $(\sigma_2,\ldots,\sigma_n)=(\sigma_2',\ldots,\sigma_n')$, hence $m|(\sigma_1-\sigma_1')$ and, since the remainders modulo $m$ in $S$ are distinct, $\sigma_1=\sigma_1'$.
\end{proof}

\section{Generalization of Kenyon's theorem}
\label{sec3}
If $q\in(0,1)$ and $\Sigma\subset \mathbb R$ is a finite set, then (along with the definition of $E(\Sigma, q)$) we also define
 $$F_n(\Sigma, q):=\left\{\sum_{i=1}^{n}\sigma_iq^i\colon (\sigma_i)\in\Sigma^n\right\} \text{ for } n\in\mathbb N.$$
\begin{thm}\label{thm7}
	Let $q\in(0,1)$ and $\Sigma\subset \mathbb R$ be a finite set such that $|\Sigma| q=1$. Denote $E=E(\Sigma,q)$, $F_n=F_n(\Sigma,q)$ and let $\mu_{\infty}$ be the associated Borel measure. Then the following conditions are equivalent.
	\begin{enumerate}[label=(\roman*)]
    \itemsep2pt
		\item \label{con1} $E$ has positive Lebesgue measure.
		\item \label{con2} $E$ contains an interval.
		\item \label{con3} $E$ is regularly closed.
		\item \label{con4} There exists a countable periodic $E$-tiling $\mathcal T$ of $\mathbb R$ and $n\in\mathbb N$ such that $d q^{-n}$ is a period of $\mathcal T$ for all $d \in D(\Sigma)$.
		\item \label{con5} $D(\Sigma)$ has a common divisor and $|F_n|=|\Sigma|^n$ for all $n\in\mathbb N$.
		\item \label{con6} $D(\Sigma)$ has a common divisor and $E$ has Hausdorff dimension $1$.
    \item \label{con10} 
    $\mu_\infty$ is absolutely continuous with respect to Lebesgue measure.
        \item \label{con7} There exists $\delta>0$ such that $f_\delta \# \mu_\infty$ is the Lebesgue measure on $\mathcal B[0,1)$, where $f_\delta(x)= \frac x\delta - \lfloor\frac x\delta\rfloor$ for $x\in\mathbb R$.
		\item \label{con8} There exists $\delta>0$ such that for all $n\in \mathbb N$ we have  
  $$\int_{\mathbb R}\exp\left(\frac{2\pi i n s}{ \delta}\right)\mathrm{d}\mu_\infty(s)=0.$$
		\item \label{con9} There exists $\delta>0$ such that for any $n\in\mathbb N$ there exists $k\in\mathbb N$ such that  $$\sum_{\sigma\in\Sigma}\exp\left(\frac{2\pi i 
 n \sigma q^k}{ \delta}\right)=0.$$
	\end{enumerate}
 Moreover, if the above conditions hold, then \begin{equation}\label{eq2}\mu_\infty(B)=\frac{\lambda(B\cap E)}{\lambda(E)} \text{ for all }B\in\mathcal B(\mathbb R).\end{equation}
\end{thm}
\begin{proof} First we note that if $\lambda(E)>0$, then $|F_n|=|\Sigma|^n$ and the sets $v+q^n E$ are Lebesgue measure disjoint for distinct $v\in F_n$. Indeed, we have $E=F_n +q^n E$ and hence \begin{multline*}
		\lambda(E)=\lambda\left(\bigcup_{v\in F_n}(v+q^n E)\right)\leqslant \sum_{v\in F_n}\lambda(v+q^n E)=\\=|F_n| q^n\lambda(E)\leqslant|\Sigma|^n  q^n\lambda(E)= \lambda(E).	\end{multline*}
	\begin{enumerate}
		\item[\ref{con1}$\Rightarrow$\ref{con2}.] Let $x\in E$ be a point of density for Lebesgue measure. For each $n\in\mathbb N$ consider the sets
		$$V_n=\left\{v\in  F_n\colon \left(vq^{-n} + E\right)\cap \left[xq^{-n}-1, xq^{-n}+1\right]\neq\emptyset\right\},$$
		$$U_n=q^{-n}(V_n-x).$$
If $u=q^{-n}(v-x)\in U_n$, then $(u+E)\cap [-1,1]\neq \emptyset$. Hence $$U_n+E\subset [-1-\mathrm{diam}\, E, 1+\mathrm{diam}\, E] \text{ and}$$
$$U_n\subset [-1-\mathrm{diam}\, E, 1+\mathrm{diam}\, E]+\min E=:[a,b].$$
		
		Since $u+E$ are measure disjoint for distinct $u\in U_n$, so $|U_n|\leqslant  k$, where $k=\lceil(2+2\mathrm{diam}\,E)/\lambda(E)\rceil$. It is important that $k$ is independent of $n$.
		
		Consider $u^n\in [a,b]^k$ such that $U_n=\{u^n_1,\ldots, u^n_k\}$ for all $n\in\mathbb N$. By compactness of $[a,b]^k$ choose a subsequence $(u^{n_i})_i$ convergent to $u\in[a,b]^k$ and let $U:=\{u_1,\ldots, u_k\}$.
		
		Since $x$ was a density point, and by Lemma 1, we have
		\begin{equation*}
			\begin{split}
				2 & =\lim_{n\to\infty}\lambda\left(q^{-n}E\cap\left[xq^{-n}-1, xq^{-n}+1\right]\right)=\\&=\lim_{n\to\infty}\lambda\left(\left(q^{-n}V_n+E\right)\cap\left[xq^{-n}-1, xq^{-n}+1\right]\right)=\\&=
				\lim_{n\to\infty}\lambda((U_n+E)\cap[-1,1])=\\&=\lim_{n\to\infty}\lambda\left(\bigcup_{j=1}^k(u^n_j+E)\cap [-1,1]\right)=\\&= \lambda\left((U+E)\cap[-1,1]\right).
			\end{split}
		\end{equation*}
		
		Since $U+E$ is closed, we get $[-1,1]\subset U+E$. 
		Note that $E$, as a closed set, is either nowhere dense or has non-empty interior. Since a finite union of nowhere dense sets is nowhere dense, $E$ contains an interval.

	\item[\ref{con2}$\Rightarrow$\ref{con3}]  Fix $\varepsilon>0$ and $y=\sum_{i=1}^\infty\sigma_iq^i\in E$, where $(\sigma_i)\in\Sigma^{\mathbb N}$. Choose $n\in\mathbb N$ such that $q^n\diam E<\varepsilon$. Then $\sum_{i=1}^{n}\sigma_iq^i\in F_n$. Observe that $E=\bigcup_{v\in F_n}g_v[E]$ where $g_v(x)=v+q^nx$ are affine functions. Thus if $a$ is an interior point of $E$, so is $g_v(a)$. Therefore  
	$$\sum_{i=1}^{n}\sigma_iq^i+q^na\in \inte E.$$
	Moreover
	$$\left|\sum_{i=1}^{n}\sigma_iq^i+q^na-y\right|=q^n\left|a-\sum_{i=1}^\infty\sigma_{i+n}q^{i}\right|\leqslant q^n\diam E<\varepsilon.$$ Hence $E$ is the closure of its interior points, or regularly closed.
	\item[\ref{con3}$\Rightarrow$\ref{con4}] First we need to make some observations.
	
	\noindent\textbf{Claim 1.} \textit{Given $m\in\mathbb N$, $v\in\mathbb R$ and a bounded interval $I\subset \mathbb R$ there exists a finite $E$-tiling of $I$ such that its basis $B$ contains $v+q^{-m} F_m$.}
	
	\noindent \textit{Proof.} Fix $m\in\mathbb N$,  $v\in\mathbb R$ and a bounded interval $I\subset \mathbb R$. Without loss of generality assume that $v+q^{-m} F_m\subset I$. Define $$L:=\diam (q^{-m}  F_m ) +\max\{|e|\colon e\in E\}$$
	and $[x,y]:=[\inf I- L,\sup I +L]$.
	Take an interval $[a,b]\subset E$ and find $n\in\mathbb N$ such that $n> m$ and $$(b-a)>q^n(y-x).$$ 
    We have 
	$$
	q^{-n}[a,b]\subset q^{-n}E=q^{-n}( F_n+q^nE)=q^{-n} F_n+E.
	$$
Put $F:=x-aq^{-n}+q^{-n} F_{n-m}$. Then, since $q^{-n} F_n=q^{-n} F_{n-m}+q^{-m} F_m$,
\begin{multline*}[x,y]\subset x+[0,(b-a)q^{-n}]=\\=x-aq^{-n}+q^{-n}[a,b]\subset x-aq^{-n}+q^{-n} F_n+E=\\=F+q^{-m}F_m+E.\end{multline*}
     
	Since $$v+q^{-m} F_m\subset I\subset [x,y]\subset  F+q^{-m} F_m+E,$$ there exists $f\in F$ such that $$|v-f|\leqslant \diam (q^{-m}  F_m ) +\max\{|e|\colon e\in E\}=L. $$ Put 
	$$B:=F+q^{-m} F_m+v-f.$$ 
	Since $f\in F$, so $v+q^{-m} F_m \subset B$ and also $$I\subset[x,y]+v-f\subset F+q^{-m}F_m+E+v-f=B+E.$$ 
	\begin{flushright}
	    \qedsymbol
	\end{flushright}
    
	\noindent\textbf{Claim 2.} \textit{Let $\mathcal T_0$ be an  
		$E$-tiling (with the basis $B_0$) of the interval $[x,y]$, where $y-x \geqslant \diam E$. If $\mathcal T$ is an $E$-tiling (with the basis $B$) of $[u,w]\supset [x,y]$ extending $\mathcal T_0$, then  
		$$\{\beta\in B\colon (\beta+E)\cap (u,w)\neq \emptyset\}$$ is uniquely determined by $\mathcal T_{0}$. In particular, an $E$-tiling of $\mathbb R$ extending $\mathcal T_0$ is unique.}

	\noindent \textit{Proof.}
    First we show that if $\beta\in B\setminus B_0$, then either $\beta+E < (x,y)$ or $ (x,y)<\beta+ E$. 
	Suppose that this is not the case. Since $y-x\geq\diam E$, there is a point $z\in(x,y)\cap(\beta+E)$. As $E$ is regularly closed, $z$ is in the closure of $\beta+\inte E$, and therefore $(x,y)\cap(\beta+E)$ contains an open interval. But this contradicts the fact that $(x,y)\subset B_0+E$ and $\beta+E$ is measure disjoint with $B_0+E$.

	Define
	$$B_+=\{\beta\in B\setminus B_0\colon (\beta+E)\cap (u,w)\neq \emptyset\text{ and } \beta+ E>(x,y)\},$$
	$$B_-=\{\beta\in B\setminus B_0\colon (\beta+E)\cap (u,w)\neq \emptyset\text{ and }  \beta+ E<(x,y)\}.$$
	Consider the maximal interval $I_0=[p_0,q_0]$ of $B_0+E$ such that $[x,y]\subset I_0$. If $q_0\geqslant w$, then $B_+$ is empty (and hence uniquely determined). Assume that $q_0< w$. 
	Our aim is to show that $q_0=\min B_++\min E$. By maximality of $I_0$ there is a strictly decreasing sequence $(a_n)$ tending to $q_0$ such that $a_n\notin B_0+E$. Since $B_0+E$ is closed, there are pairwise disjoint open intervals $J_n$ centered in $a_n$ and disjoint with $B_0+E$. If $q_0<\min B_++\min E$, then $B+E$ could no cover any $J_n$ that is contained in $(q_0,\min B_++\min E)$. This contradicts the fact that $\mathcal{T}$ is a cover of $[u,w]$.  If $q_0>\min B_++\min E$, then $(\min B_++\inte E)\cap[p_0,q_0]$ would contain an interval, as $E$ is regularly closed. This again yields a contradiction, this time with the fact that elements of $\mathcal{T}$ are pairwise measure disjoint. Note that $\beta_0:=q_0-\min E$ depends only on $\mathcal{T}_0$. Then $\beta_0=\min B_+$, and therefore the smallest element of $B_+$ is uniquely determined by $\mathcal{T}_0$.

	In the next step consider the maximal interval $I_1=[p_1,q_1]$ of $B_0\cup\{\beta_0\}+E$ such that $[x,y]\subset I_1$. If $q_1\geqslant w$, then $B_+$ has exactly one element $\beta_0$. If $q_1<w$, then by the same reasoning as before we get $q_1=\min(B_+\setminus\{\beta_0\})+\min E=\beta_1+\min E$.
	
	We continue this procedure inductively. The induction stops at the point when $q_n\geq w$. It cannot go on infinitely because that would yield an increasing and bounded sequence $\{\beta_0,\beta_1,\ldots\}\subset B$ and infinitely many measure disjoint copies of $E$ would be contained in a bounded interval, which is impossible because of positive measure.
    
	Therefore $B_+=\{\beta_0,\beta_1,\ldots, \beta_k\}$ is a (finite) uniquely determined set.
	We proceed analogously with $B_-$. \qedsymbol
	
	\noindent\textbf{Claim 3.} \textit{Given $v\in\mathbb R$ and $n\in\mathbb N$, if
		$[x,y]\subset v+q^{-n} F_n +E$ for some interval $[x,y]$ of length $\geqslant \diam E$, then there exists a unique  countable $E$-tiling of $\mathbb R$ with the basis $B$ such that $v+q^{-n} F_n\subset B$.}
	
	\noindent \textit{Proof.} Fix $v\in\mathbb R$ and $n\in\mathbb N$ such that
	$[x,y]\subset v+q^{-n} F_n+E$ for some interval $[x,y]$ of length $\geqslant \diam E$.
	By Claim 1, for any $M>0$ there exists a finite $E$-tiling of $[-M,M]$ with the basis $B_M\supset v + q^{-n} F_n$. Put $B=\bigcup_{M>0}B_M$.
	It follows from Claim 2 that $\mathcal T= \{\beta+ E\colon\beta\in B\}$ is a desired unique $E$-tiling of $\mathbb R$. \qedsymbol
	
	Condition (iv) follows from the following stronger Claim 4, which we will need later on.
 
\noindent\textbf{Claim 4.} \textit{For  any sequence of digits $(\sigma_i)\in\Sigma^{\mathbb N}$, there exists a periodic $E$-tiling of $\mathbb R$ with the basis $B$ such that $$B=\bigcup_{n\in\mathbb N}\left(q^{-n}F_n-\sum_{i=1}^{n}\sigma_i q^{-i+1}\right)+t\mathbb Z,$$ where $t\neq 0$ is any period of $B$. Moreover, for sufficiently large $n\in\mathbb N$, the numbers $d q^{-n}$ are periods of $B$ for all $d \in D(\Sigma)$.}
		
		\noindent \textit{Proof.}  Fix a sequence of digits $(\sigma_i)\in\Sigma^{\mathbb N}$ and denote $$A_n:=q^{-n}F_n-\sum_{i=1}^{n}\sigma_i q^{-i+1}$$ for $n\in\mathbb N$. Take an interval $[a,b]\subset E$ and fix $n\in\mathbb N$ such that $b-a> q^{n} \diam E$. Since for any fixed $v\in\mathbb R$ we have $$v+q^{-n}[a,b]\subset v+ q^{-n}E= v+ q^{-n}F_n+E,$$ by Claim 3, {applied for $v=-\sum_{i=1}^{n}\sigma_iq^{-i+1}$,} there exists a unique $E$-tiling of $\mathbb R$ with the basis $B_n$ such that $A_n\subset B_n$. Since the sequence of sets $(A_n)$ is increasing with respect to $n$, it follows from uniqueness that $B_n=B$ does not depend on $n$.
		
		Fix a period $t\neq 0$.  
		Take an interval $[a,b]\subset E$ and $m\in\mathbb N$ such that $b-a> q^m t$. Then
		$$q^{-m}[a,b]\subset q^{-m}F_m+E$$
		and hence $A_m+E$ contains an interval $[x,y]$ of length $>t$. Fix $\beta\in B$.
		Then there exists $k\in\mathbb Z$ such that $$\emptyset\neq(\beta+kt+\inte E)\cap (x,y)\subset (\beta+kt+\inte E)\cap (A_m+\inte E).$$ 
		Since $\beta+kt\in B$, $A_m\subset B$ and distinct elements of the tiling are measure disjoint, it follows that $\beta+kt\in A_m$, that is, $$\beta\in A_m-kt\subset \bigcup_{m\in\mathbb N}A_m+ t\mathbb Z.$$ The opposite inclusion is obvious.
  
	For any $v\in \mathbb R$ we have 
	$$v+q^{-n}[a,b]\subset v+ q^{-n}E= v+ q^{-n}F_n+E,$$
	and so, by Claim 3 (recall that $q^n\diam E<b-a$), there exists a unique $E$-tiling of $\mathbb R$ with the basis $B_v\supset v+ q^{-n}F_n$.
	Note that for any $v$ and $v'$ we have $v'+q^{-n}F_n\subset B_v+(v'-v)$, so by uniqueness we get $B_v+(v'-v)=B_{v'}$. 
        Since $$q^{-n-1}F_{n+1}=q^{-n}\Sigma+q^{-n}F_n,$$
	 we get $B_v=B_{n+1}=B$ for any $v\in q^{-n}\Sigma-\sum_{i=1}^{n+1}\sigma_{i}q^{-i+1}$. In particular, for any $\sigma,\sigma'\in\Sigma$ we have $(\sigma-\sigma')q^{-n}+B=B$, that is, $(\sigma-\sigma')q^{-n}$ is a period of $\mathcal T$. \qedsymbol
	
	\item[\ref{con4}$\Rightarrow$\ref{con5}] Let $B\subset \mathbb R$ be the basis of the $E$-tiling $\mathcal T$. Since $B$ is countable, it follows from Baire category theorem that $E$ contains an interval, and hence (as noted at the beginning of the proof) $|F_n|=|\Sigma|^n$. The set $P$ of all periods of $B$ is closed under subtraction. Any nontrivial subgroup $G$ of $(\mathbb R,+)$ is either dense or has the form $G=\delta\mathbb Z$ for some $\delta>0$. If $P$ were dense, so would be $B$ and some bounded interval would contain infinitely many elements of $\mathcal T$, which contradicts the fact that they are measure disjoint. Hence $P=\delta\mathbb Z$ for some $\delta>0$ and, by assumption, $D(\Sigma)\subset q^nP=(\delta q^n) \mathbb Z$.

 \item[\ref{con5}$\Rightarrow$\ref{con10}] Let $\delta>0$ be a common divisor of $D(\Sigma)$. Since $|F_n|=|\Sigma|^n$, we have
 $$\mu_n(B)=q^n|B\cap F_n|$$ for each $n\in\mathbb N$ and Borel set $B\in\mathcal B(\mathbb R)$. Note that $F_n-F_n\subset (\delta q^{n})\mathbb Z$. It follows that for any interval $I\subset \mathbb R$ we have
\[\mu_n(I)=q^n|I\cap F_n|\leqslant q^nk < \frac{\lambda(I)}\delta+ q^n,\]
where $k\in\mathbb N$ is chosen in such a way that $(k-1)\delta q^n<{\lambda(I)}\leqslant k\delta q^n$. Therefore
\[
\mu_\infty(I)\leqslant\liminf_{n\to\infty}\mu_n(I)\leq\frac{\lambda(I)}{\delta}
\] for any open interval $I\subset \mathbb R$. Now for any Borel set $B\in\mathcal B(\mathbb R)$ we have
\begin{multline*}\lambda(B)=\inf\left\{\sum_{k\in\mathbb N}\lambda(I_k)\colon B\subset \bigcup_{k\in\mathbb N}I_k\right\}\geqslant \\ \geqslant \inf\left\{\sum_{k\in\mathbb N}\delta\mu_\infty(I_k)\colon B\subset \bigcup_{k\in\mathbb N}I_k\right\} \geqslant \delta\mu_{\infty}(B).\end{multline*}

\item[\ref{con1}$\Rightarrow$\ref{con6}] One-dimensional Hausdorff measure in $\mathbb R$ coincides with Lebesgue measure.

\item[\ref{con6}$\Rightarrow$\ref{con5}] Suppose $|F_n|=k<|\Sigma|^n$ for some $n\in\mathbb N$. Then for all $m\in\mathbb N$ we have $|F_{mn}|\leqslant |F_n|^m=k^m$ and for any $\varepsilon>0$ there exists $m_{\varepsilon}\in\mathbb N$ such that $$q^{m_{\varepsilon}n}\diam E <\varepsilon.$$
Since $E=F_{m_{\varepsilon}n}+q^{m_{\varepsilon}n}E$, for any $s>0$ we have
$$ 
H^s(E)\leqslant\lim_{\varepsilon\to 0}|F_{m_\varepsilon n}|\left(q^{m_{\varepsilon} n}\diam E\right)^s\leqslant\lim_{m\to\infty}(kq^{ns})^m(\diam E)^s,$$
where $H^s$ denotes $s$-dimensional Hausdorff measure. Since $$kq^{ns}<1 \iff s>\frac1n\log_{|\Sigma|}(k),$$ it follows that $$\dim_H(E)\leqslant \frac1n\log_{|\Sigma|}(k)<1.$$

\item[\ref{con10}$\Rightarrow$\ref{con1}] $E$ is the support of measure $\mu_\infty$.

\item[\ref{con10}$\Rightarrow$\ref{con7}]
Consider the space $X=[0,1)$ with the circle topology and
define $T\colon [0,1)\rightarrow [0,1)$ by the formula
$$T(x)=|\Sigma|x\, \mathrm{mod}\, 1.$$
Then $T$ is a continuous map that is measure preserving and ergodic with respect to Lebesgue measure. Let $\delta$ be the common divisor of $D(\Sigma)$. We can assume without loss of generality that $0\in\Sigma$. Indeed, if $\mu_\infty^x$ is the measure asssociated with $\Sigma+x$ and $q$, then \begin{multline*}
    f_\delta\#\mu_\infty^x(B)=\mu_\infty\left(f_\delta^{-1}[B]-\frac{x}{1-q}\right)=\\=\mu_\infty\left(f_\delta^{-1}\left[\left(B-\frac x{\delta(1-q)}\right)\operatorname{mod}1\right]\right).
\end{multline*}
We have
\begin{multline*}
    T\#(f_\delta\#\mu_{n+1})(B)=q\sum_{\sigma\in\Sigma}\mu_n\left(\frac1q (T\circ f_\delta)^{-1}[B]-\sigma\right)=\\=q\sum_{\sigma\in\Sigma}\mu_n\left( (T\circ f_\delta\circ h_\sigma)^{-1}[B]\right)=q\sum_{\sigma\in \Sigma}\mu_n\left(f_\delta^{-1}[B]\right)=f_\delta\# \mu_n(B),
\end{multline*}
where $h_\sigma(x)=q(x+\sigma)$ for $x\in \mathbb R$ and $\sigma\in \Sigma$. Since $T$ and $f_\delta$ are continuous (with respect to circle topology), we get $T\# (f_\delta \# \mu_\infty)=f_\delta\# \mu_\infty$, that is, $T$ is measure preserving with respect to $f_\delta\# \mu_{\infty}$.

Since $f_\delta^{-1}[A]=\bigcup_{k\in\mathbb Z}(A+k)\delta$, it follows from the assumption that $f_\delta\#\mu_\infty$ is absolutely continuous with respect to Lebesgue measure. Hence $T$ is ergodic with respect to $f_\delta\#\mu_\infty$. By Lemma \ref{lem3} the measure $f_\delta\#\mu_\infty$ is Lebesgue measure.

\item[\ref{con7}$\Rightarrow$\ref{con10}] Clearly,
 $$\lambda(B)\geqslant \delta\lambda (f_\delta[B])=\delta\mu_\infty(f_\delta^{-1}[f_\delta[B]])\geq \delta\mu_\infty(B)$$
for all sets $B\in\mathcal B(\mathbb R)$.
\item[\ref{con7}$\Rightarrow$\ref{con8}] 
For any $n\in\mathbb N$ we have
\begin{multline*}
    0=\int_0^1\exp\left({2\pi i n s}\right)\mathrm{d}\lambda(s)=\int_0^1\exp\left({2\pi i n s}\right)\mathrm{d}(f_\delta\# \mu_\infty)(s)=\\=\int_{\mathbb R}\exp\left(2\pi i n f_\delta(s)\right)\mathrm{d}\mu_{\infty}(s)=\int_{\mathbb R}\exp\left(\frac{2\pi i n  s}{\delta}\right)\mathrm{d}\mu_\infty(s).
\end{multline*}

\item[\ref{con8}$\Rightarrow$\ref{con9}] For any $n,k\in\mathbb N$ we have
			\begin{multline*}
				\int_{\mathbb R}\exp\left(\frac{2\pi i n s}\delta\right)\mathrm{d}\, \mu_{k+1}(s)=q^{k+1}\sum_{(\sigma_i)\in \Sigma^{k+1}}\exp\left(\frac{2\pi i n \sum_{i=1}^{k+1}(\sigma_iq^i)}\delta \right)=\\=q^{k+1}\sum_{(\sigma_i)\in \Sigma^{k}}\sum_{\sigma\in\Sigma}\exp\left(\frac{2\pi i n \sum_{i=1}^{k}(\sigma_iq^i)}\delta \right)\exp\left(\frac{2\pi i n \sigma q^k}{\delta}\right)=\\=q\sum_{\sigma\in\Sigma}\exp\left(\frac{2\pi i n \sigma q^k}{\delta}\right)\int_{\mathbb R}\exp\left(\frac{2\pi i n s}\delta\right)\mathrm{d}\, \mu_{k}(s).
			\end{multline*}
			It follows by simple induction that
			
			\begin{equation}\label{eq1}\int_{\mathbb R}\exp\left(\frac{2\pi i n s}\delta\right)\mathrm{d}\, \mu_{\infty}(s)=\prod_{k=1}^\infty q\sum_{\sigma\in\Sigma}\exp\left(\frac{2\pi i n \sigma q^k}{\delta}\right).\end{equation}
			
			Fix $n\in\mathbb N$ and suppose  $$a_k:=q\sum_{\sigma\in \Sigma}\exp\left(\frac{2\pi i n \sigma q^k}{\delta}\right)\neq 0$$ for all $k\in\mathbb N$. For a finite set $S\subset \mathbb R$ and $\alpha\in (0,\pi)$ such that $\max\{|2\pi s|\colon s\in S\}< \alpha$  we have $$\left|\sum_{s\in S}\exp\left(2\pi i s\right)\right|\geqslant |S|\cos\left(\alpha\right).$$
			Hence for sufficiently large $k\in \mathbb N$ we have
			$$\log\left(q\left|\sum_{\sigma\in\Sigma}\exp\left(\frac{2\pi i n \sigma q^k}{\delta}\right)\right|\right)\geqslant \log\left(q|\Sigma|\cos\left(\frac1 k\right)\right)=\log(\cos\left(\frac1 k\right)).$$
			Since the series $$\sum_{k=1}^{\infty}\log(\cos\left(\frac1k\right))$$ is convergent, it follows that
			$$\prod_{k=1}^\infty |a_k| > 0,$$
			which contradicts our assumption.
			
			\noindent\textbf{Remark.} In the derivation of the equation \eqref{eq1} we did not use the assumption (ix).
			
	\item[\ref{con9}$\Rightarrow$\ref{con8}] It follows immediately from the equation \eqref{eq1}. 

        \item[\ref{con8}$\Rightarrow$\ref{con7}] Define $f\colon \mathbb R\rightarrow [0,1)$ as
        $f(x)=\frac x\delta -\lfloor \frac x \delta \rfloor$ and let $\mu_\infty'\colon \mathcal B[0,1)\rightarrow [0,1]$ be given by the formula $\mu_\infty'(B)=\mu_\infty(f^{-1}[B])$, that is, $\mu_\infty'=f\# \mu_\infty$. Then for any $n\in\mathbb N$ we have
        \begin{multline*}\int_0^1\exp(2\pi i n s)\mathrm{d}\, \mu_\infty'(s)=\int_{\mathbb {R}}\exp(2\pi i n \left(\frac x\delta -\left\lfloor \frac x \delta \right\rfloor\right))\mathrm{d}\, \mu_\infty(s)=\\=\int_{\mathbb R}\exp\left(\frac{2\pi i n s}\delta\right)\mathrm{d}\, \mu_{\infty}(s)=0,\end{multline*} that is, all Fourier coefficients $\hat{\mu}_\infty'(n)$ (except for $\hat{\mu}_\infty'(0)$) are zeros. It follows from Lemma \ref{lem4} that $\mu_\infty'$ is Lebesgue measure $\lambda|_{\mathcal B[0,1)}$. 
 \end{enumerate}

Now we prove \eqref{eq2}. Since a translation of $\Sigma$ by $x$ results in translating $E$ and a shift in $\mu_\infty$ by the same number $\frac{x}{1-q}$, we can assume without loss of generality that $0 \in \Sigma$.
    By Claim 4 applied for the sequence $(\sigma_i)\equiv 0$, there exists some $E$-tiling of $\mathbb R$ with the basis $B$ such that $$\bigcup_{n\in\mathbb N}q^{-n}F_n\subset B.$$ Define
$$s(M)=\sup\left\{\frac{|B\cap I|}{\lambda(I)}\colon I \text{ is an open interval of length } \geqslant M\right\}$$ and
$$u:=\lim_{M\to\infty}s(M).$$
Then for any open interval $I\subseteq \mathbb R$ we have
$$|B\cap I|\cdot\lambda(E)=\sum_{x\in B\cap I}\lambda(x+E)=\lambda\left(\bigcup_{x\in B\cap I}(x+E)\right)\leqslant \lambda(I)+2\diam E.$$
Hence $$s(M)\leqslant \frac1{\lambda(E)} + \frac{2\diam E}{\lambda(E)M}$$ and consequently $u\leqslant \frac1{\lambda(E)}$. For any open interval $I$ and $n\in\mathbb N$ we have
\begin{multline*}
    \mu_n(I)=q^n|I\cap F_n|=q^n|(q^{-n}I)\cap (q^{-n}F_n)|\leqslant \\ \leqslant q^n |(q^{-n}I)\cap B|\leqslant q^ns(\lambda_n)\lambda(q^{-n}I)=s(\lambda_n)\lambda(I),
\end{multline*}
where $\lambda_n=\lambda(q^{-n}I)$. Therefore
$$\mu_\infty(I)\leqslant \liminf_{n\to\infty} \mu_n(I) \leqslant u\lambda(I)\leqslant \frac{\lambda(I)}{\lambda(E)}.$$ Now for any Borel set $B\in\mathcal B(\mathbb R)$ we have
\begin{multline*}\lambda(B)=\inf\left\{\sum_{k\in\mathbb N}\lambda(I_k)\colon B\subset \bigcup_{k\in\mathbb N}I_k\right\}\geqslant \\ \geqslant \inf\left\{\sum_{k\in\mathbb N}\lambda(E)\mu_\infty(I_k)\colon B\subset \bigcup_{k\in\mathbb N}I_k\right\} \geqslant \lambda(E)\mu_{\infty}(B)\end{multline*}
and consequently
    $$ 1=\mu_\infty(E)=\mu_\infty( B)+\mu_\infty(B^c)\leqslant\frac{\lambda(E\cap B)+\lambda(E\setminus B)}{\lambda(E)}=1.$$
	\end{proof}

The implication (i)$\Rightarrow$(ii) was actually already proven by a different method in \cite[2.3 Corollary]{Schief}.

\begin{cor}
    Assume that the set $(\Sigma - x)/\delta$ has all remainders modulo $|\Sigma|$ for some $x\in\mathbb R$ and $\delta>0$. Then $B=\delta\mathbb Z$ is the basis for the $E$-tiling of $\mathbb R$, $\delta$ is  the greatest common divisor of $D(\Sigma)$ and $\lambda(E)=\delta$, where $E=E(\Sigma,q)$ and $|\Sigma|\cdot q=1$.
\end{cor}
\begin{proof}
    Clearly, condition (x) of Theorem \ref{thm7} is satisfied with $\delta$.
     Since a translation of $\Sigma$ by $x$ results in translating $E$ and a shift in $\mu_\infty$ by the same number $\frac{x}{1-q}$, we can assume without loss of generality that $x= 0 \in \Sigma$.
    Also note that $\delta$ is the greatest common divisor of $D(\Sigma)$.
    
    By Claim 4 applied for the sequence $(\sigma_i)\equiv 0$, there exists a periodic $E$-tiling of $\mathbb R$ with the basis $B$ such that $$B=\bigcup_{k\in\mathbb N}q^{-k}F_k+t\mathbb Z,$$ where $t\neq 0$ is any period of $B$, and for some $n\in\mathbb N$ the numbers $d q^{-n}$ are periods for all $d \in D(\Sigma)$. Since there exist  $k_1,\ldots, k_m\in\mathbb Z$ and $d_1,\ldots, d_m\in D(\Sigma)$ such that $k_1d_1+\dots+k_md_m=\delta$, the number $\delta q^{-n}$ is also a period of $B$.
    
    Define $C:=B\cap [0,\delta q^{-n})\subset \delta \mathbb Z$. Then, since $q^{-k-1}F_k+\Sigma=q^{-k-1}F_{k+1}$, for any $\gamma \in \Sigma/\delta$ the set
    $C/\delta$ is closed under the map $$f_\gamma(x)=(x/q+\gamma) \,\mathrm{mod}\, q^{-n}.$$ 
    
    Note that, by Lemma \ref{lem2} with $S=\Sigma/\delta$, if $(\gamma_1,\ldots, \gamma_{n})\neq (\gamma_1',\ldots, \gamma_{n}')$, then $$(f_{\gamma_1}\circ\dots\circ f_{\gamma_{n}})(x)\neq (f_{\gamma_1'}\circ\dots\circ f_{\gamma_{n}'})(x).$$ It follows that
    $$\{0,1,\ldots, q^{-n}-1\} \subset C/\delta,$$
    and in fact $B=\delta\mathbb Z$. 
    
    Let $E_k=E\cap\delta[k,k+1)$ for $k\in\mathbb Z$. Then $$\lambda\left((E_k-k\delta)\cap(E_l-l\delta)\right)\leqslant \lambda ((E-k\delta)\cap (E-l\delta))=\lambda(E\cap (E+\delta(k-l)))=0,$$ provided $k\neq l$. Hence \begin{multline*}
        \lambda(E)=\lambda\left(\bigcup_{k\in\mathbb Z}E_k\right)=\sum_{k\in\mathbb Z}\lambda(E_k)=\\=\sum_{k\in\mathbb Z}\lambda(E_k-k\delta)=\lambda\left(\bigcup_{k\in\mathbb Z}(E_k- k\delta)\right)=\lambda([0,\delta))=\delta.
    \end{multline*}
\end{proof}

\section{Prime case}
\label{sec4}

\begin{lem}
    \label{lem5}Let $p$ be an odd prime number and $k\in\mathbb N$. If $\Sigma\subset \mathbb Z$ is a finite set of cardinality $p$ and $$\sum_{s\in \Sigma}\exp\left(\frac{2\pi i s }{p^k}\right)=0,$$
    then $\left\{\exp\left(\frac{2\pi i s }{p^k}\right)\colon s\in \Sigma\right\}$ forms the set of vertices of a regular polygon.
\end{lem}
\begin{proof}
    Let $\Sigma \operatorname{mod}p^k=\{s_0,\ldots,s_{p-1}\}$, where $s_0\leq\ldots\leq s_{p-1}$. We choose $j$ such that $$s_{j+1}-s_j=\max\{s_{i+1}-s_{i}\colon i=0,\ldots, p-1\}\geq p^{k-1},$$ where $s_p:=s_0$. 
    We  assume that $\max(\Sigma \operatorname{mod}p^k)=p^k-p^{k-1}$, because we  can translate $\Sigma$ (by an integer) in such a way that $s_j$ is carried to $p^k-p^{k-1}$.

    Let $\varepsilon=\exp\left(\frac{2\pi i}{p^k}\right)$. By assumption $$0=\varepsilon^{s_0}+\dots+\varepsilon^{s_{p-1}},$$ that is, $\varepsilon$ is the root of a polynomial$$X^{s_0}+\dots+X^{(p-1)p^{k-1}}.$$ On the other hand, the minimal polynomial of $\varepsilon$ over $\mathbb Q$ is $$1+X^{p^{k-1}}+X^{2p^{k-1}}+\dots+X^{(p-1)p^{k-1}},$$
    cf. \cite[p. 280]{Lang}.
\end{proof}
\begin{cor}
    Let $|\Sigma|\cdot q=1$, where $|\Sigma|=p$ is a prime number. Then $E(\Sigma,q)$ contains an interval if and only if $(\Sigma-x)/\delta$ has all remainders modulo $p$ for some $x\in\bbR$ and $\delta>0$. 
\end{cor}
\begin{proof}
    Assume that $E(\Sigma,q)$ contains an interval. By Theorem \ref{thm7} there exists $\delta>0$ such that for any $n\in\mathbb N$ there exists $k\in\mathbb N$ such that  $$\sum_{\sigma\in\Sigma}\exp\left(\frac{2\pi i 
 n \sigma q^k}{ \delta}\right)=0.$$
It follows from the proof of Theorem \ref{thm7} that $\delta$ can be chosen to be the greatest common divisor of $D(\Sigma)$. Define $\Sigma^*=(\Sigma-x)/\delta$, where $x=\min \Sigma$. Then $\Sigma^*\subset \mathbb Z$ and there exists $k\in\mathbb N$ such that
$$\sum_{s\in\Sigma^*}\exp\left(\frac{2\pi i s }{p^k}\right)=0.$$
By Lemma \ref{lem5},  $$\left\{\exp\left(\frac{2\pi i s }{p^k}\right)\colon s\in \Sigma^*\right\}=\left\{\exp\left(\frac{2\pi i t}{p}\right)\colon t\in\{0,1,\ldots, p-1\}\right\},$$
and $$\Sigma^* \operatorname{mod} p =p^{k-1}\{0,1,\ldots,p-1\}.$$
Since members of $\Sigma^*$ are relatively prime, we have in fact $k=1$.

The opposite implication follows from Theorem \ref{thm7}, since condition \ref{con9} is satisfied.
\end{proof}

\section{Final remarks}

\subsection{Achievement sets for multigeometric series}

Let $(k_1,\dots,k_m;\frac{1}{n})$ be a multigeometric series such that some $k_\ell$ is divisible by $n$, say $\frac {k_\ell}n=r\in\mathbb Z$. Then
\[
\underbrace{A\left(k_1,\dots,k_m;\frac{1}{n}\right)}_E=\{0,r\}+\underbrace{A\left(k_1',\dots,k_m';\frac{1}{n}\right)}_{E'}
\]
where $k_i'=k_i$ if $i\neq\ell$ and $k_\ell'=r$. Since $E=E'\cup rE'$, then $E$ has the same topological type as $E'$. Hence, in determining the type we may assume that none of $k_1,\dots,k_m$ is divisible by $n$. 

Let us present an example. For a sequence $(1,8;\frac14)$ the set $\Sigma$ is of the form $\{0,1,8,9\}$. Note that 
\[
(1,8;\frac14)=(\frac{1}{4},{\frac{8}{4}},\frac{1}{16},{\frac{8}{16}},\frac{1}{64},{\frac{8}{64}},\dots)=
\]
\[
(\frac{1}{4},2,\frac{1}{16},{\frac{1}{2}},\frac{1}{64},{\frac{1}{8}},\dots)
\]
After rearrangement of terms we obtain
\[
(2,{\frac{1}{2}},\frac{1}{4},{\frac{1}{8}},\frac{1}{16},\dots)
\]
Therefore $A(1,8;\frac14)=[0,1]\cup[2,3]$. This example shows that Nitecki Theorem cannot be reversed. However, 8 is divisible by 4. Now, we are ready to formulate our conjecture.  
\begin{prob}
Let $(k_1,\dots,k_m;\frac{1}{n})$ be a mutigeometric series such that $k_i\in\mathbb N$, $\Sigma=\{\sigma_0<\dots<\sigma_{n-1}\}$ has $n$ elements, none of $k_1,\dots,k_m$ is divisible by $n$ and $E(k_1,\dots,k_m;\frac{1}{n})$ has positive Lebesgue measure. Is is true that the set of remainders $t_i<n$ with $\sigma_i\equiv t_i$ equals $\{0,1,\dots,n-1\}$?    
\end{prob}

\subsection{Projections of some class of IFS attractors}
As we have mentioned in Introduction, the work \cite{K} of Richard Kenyon was motivated by studying projections of the one-dimensional Sierpi\'nski Gasket. It is the attractor $S$ in $\mathbb R^2$ for the three linear mappings $f_1(x,y)=\frac13(x,y)$, $f_2(x,y)=\frac13(x+1,y)$ and $f_3(x,y)=\frac13(x,y+1)$. Let $S_u=\pi_u(S)$ be the linear projection of $S$ onto the $x$-axis, where $\pi_u(x,y)=x+uy$. Kenyon observed that $S_u=\{\sum_{n=1}^\infty\alpha_n\frac{1}{3^n}:\alpha_n=0,1,u\}$, which in our notation is $S_u=E(\Sigma,\frac13)$ for $\Sigma=\{0,1,u\}$. In this section we make some observations using this approach.

Let $f_i:\bbR^2\to\bbR^2$, for $i=1,\dots,k$, be linear mapping of the form $f_i(x,y)=(\frac{x+a_i}{k},\frac{y+b_i}{k})$ where $a_i,b_i\in\bbR$. Let $\mathcal{F}=\{f_i\}_{i\leq k}$ be an IFS. Put
\[
S_\mathcal{F}=\{\sum_{n=1}^\infty\frac{1}{k^n}(x_n,y_n):(x_n,y_n)=(a_1,b_1),\dots,(a_k,b_k)\}
\]
Note that 
\[
f_i(S_\mathcal{F})=\frac{1}{k}S_{\mathcal{F}}+\frac{1}{k}(a_i,b_i)
\]
and therefore $S_\mathcal{F}=\bigcup_{i\leq k}f_i(S_\mathcal{F})$. This shows that $S_\mathcal{F}$ is an attractor of IFS $\mathcal{F}$. For fixed $u\in\bbR$, we have
\[
\pi_u(S_\mathcal{F})=\{\sum_{n=1}^\infty\frac{1}{k^n}(x_n+uy_n):(x_n,y_n)=(a_1,b_1),\dots,(a_k,b_k)\}=E(\Sigma_u,\frac{1}{k})
\]
where $\Sigma_u=\{a_i+ub_i:i\leq k\}$. Using Theorem \ref{thm7} one can say whether the projection $\pi_u(S_\mathcal{F})$ of an attractor $S_\mathcal{F}$ has positive measure or not. 

Note that the class of IFS attractors we have prescribed above is invariant under scalings, shiftings, and rotations. Using these operations, we may assume that $(a_1,b_1)=(0,0)$ and $(a_2,b_2)=(1,0)$. Then $\Sigma_u$ contains 0 and 1, and therefore the common divisor of  $D(\Sigma_u)$  is rational, if it exists. From that we infer that $\pi_u(S_\mathcal{F})$ has positive measure for at most countably many $u$'s. It was proved in \cite{PSS2}, that $\pi_u(S_\mathcal{F})$ has measure zero for almost all $u$. Compare also \cite{PSS}.

Consider the following example $f_1(x,y)=\frac14(x,y)$, $f_2(x,y)=\frac14(x+1,y)$, $f_3(x,y)=\frac14(x+1,y+1)$, $f_4(x,y)=\frac14(x,y+\sqrt{2})$. Then $\Sigma_u=\{0,1,1+u,u\sqrt{2}\}$ and $D(\Sigma_u)$ does not have a common divisor for any $u\in\bbR$. Thus every projection of the attractor of $\mathcal{F}=\{f_1,f_2,f_3,f_4\}$ has measure zero. 
Note that such an example does not exist for $k=3$.

\end{document}